\documentclass[11pt]{amsart}

\usepackage{amsmath,amsthm,amssymb,amscd}
\usepackage[arrow,matrix]{xy}
\usepackage[dvips]{graphicx}
\usepackage{mathrsfs}
\usepackage{enumerate}

\theoremstyle{plain}
\numberwithin{equation}{section}
\newtheorem{thm}{Theorem}[section]

\newtheorem{prop}[thm]{Proposition}

\newtheorem{lem}[thm]{Lemma}

\theoremstyle{definition}
\newtheorem{dfn}[thm]{Definition}
\newtheorem{exm}[thm]{Example}

\newtheorem{rem}[thm]{Remark}

\newtheorem{Def}{Definition}[section]

\newtheorem{Prop}[Def]{Proposition}

\def\rank{\mathop{\mathrm{rank}}\nolimits}
\def\dim{\mathop{\mathrm{dim}}\nolimits}

\def\Ad{\mathop{\mathrm{Ad}}\nolimits}

\def\gsds{\Gamma(\sigma)^{\mathrm{gp}} \backslash D_{\sigma}}

\def\GsDs{\Gamma(\sigma)^{\mathrm{gp}} \backslash D_{\sigma}}

\def\Ds{D_{\sigma}}
\def\gm{\mathbb{G}_{m}}

\def\CC{\mathbb{C}}

\def\QQ{\mathbb{Q}}
\def\RR{\mathbb{R}}
\def\ZZ{\mathbb{Z}}
\def\PP{\mathbb{P}}

\def\Gr{\mathrm{Gr}}
\def\gr{\mathrm{Gr}}

\def\calH{\mathcal{H}}

\def\frakg{\mathfrak{g}}

\def\calF{\mathcal{F}}
\def\calH{\mathcal{H}}
\def\calO{\mathcal{O}}

\def\Im{\mathop{\mathrm{Im}}\nolimits}
\def\bs{\backslash}
\def\gp{\mathrm{gp}}

\def\Aut{\mathop{\mathrm{Aut}}\nolimits}

\def\Span{\mathrm{span}}

\def\l{\left}
\def\r{\right}
\def\rt-m{\sqrt{m}}
\def\00{\mathbf{0}}

\newcommand{\C}{\mathbb{C}}

\newcommand{\Z}{\mathbb{Z}}

\newcommand{\Q}{\mathbb{Q}}

\newcommand{\mf}[1]{{\mathfrak{#1}}}

\begin{document}
\title[Degenerating Hodge structure of Calabi--Yau threefolds]{Degenerating Hodge structure of one-parameter family of Calabi--Yau threefolds} 
\author{Tatsuki Hayama\ \ \ Atsushi Kanazawa}
\date{}

\maketitle
\begin{abstract}
To a one-parameter family of Calabi--Yau threefolds, we can associate the extended period map by the log Hodge theory of Kato and Usui. 
We study the image of a maximally unipotent monodromy point under the extended period map.   
As an application, we prove the generic Torelli theorem for a large class of one-parameter families of Calabi--Yau threefolds. 
\end{abstract}


\section{Introduction}

This short article is concerned with the limit mixed Hodge structure around a maximally unipotent monodromy (MUM) point of a one-parameter family of Calabi--Yau threefolds.  
For such a family, the period domain for the Hodge structures and the limit mixed Hodge structures (LMHSs) were previously studied in \cite{KU, GGK}. 
The starting point of the present work is the theory of normalization of the LMHS around a MUM point developed in \cite{GGK}. 
The MUM points play a central role in mirror symmetry \cite{CdOGP,CK}.  
Mirror symmetry is a duality between complex geometry and symplectic geometry among several Calabi--Yau threefolds 
and it expects that each MUM point of a family of Calabi--Yau threefolds corresponds to a {\it  mirror} Calabi--Yau threefold. 
For a large class of Calabi--Yau threefolds, we observe that the normalization of the LMHS reflects the topological invariants of mirror Calabi--Yau threefolds. 

The basic idea of this article is to investigate the degenerating Hodge structures in the framework of the log Hodge theory \cite{KU}. 
An advantage of our approach is that, by slightly extending the domain and range of the period map, we obtain a better control of the period map. 
As an application, we prove the generic Torelli theorem for a large class of one-parameter families of Calabi--Yau threefolds (Theorem \ref{generic}). 
The generic Torelli theorem was confirmed for the mirror families of Calabi--Yau hypersurfaces in weighted projective spaces by Usui \cite{U2} and Shirakawa \cite{Shi}. 
Our study is a slight refinement of their technique but can be applied to a larger class of Calabi--Yau threefolds. 
The result is particularly interesting when a family has multiple MUM points and also works for new examples beyond toric geometry 
such as the mirror family of the Pfaffian--Grassmann Calabi--Yau threefolds (Example \ref{Pfaff-Grass}). 

The layout of this article is as follows. 
Section \ref{Ch Basic Hodge Theory} covers some basics of the Hodge theory and compactifications of the period domains. 
This chapter also serves to set notations. 
Section \ref{Ch Monodromy} begins with a review of the normalization of a LMHS obtained in \cite{GGK}. 
We then study the LMHSs using the normalization. 
Section \ref{Ch Generic Torelli} is devoted to the generic Torelli for a one-parameter family of Calabi--Yau threefolds. 
Section \ref{Ch MS} briefly reviews mirror symmetry with a particular emphasis on the monodromy transformation around a MUM point. 
We also discuss some suggestive examples with two MUM points. 

\subsection*{Acknowledgement}
A.K. was supported by the Harvard CMSA during this project was done. 
This research was partially supported by the Research Fund for International Young Scientists NSFC 11350110209 (Hayama). 
We are very grateful for referee's careful reading and useful suggestions. 

\section{LMHS and partial compactification of period domain} \label{Ch Basic Hodge Theory}


\subsection{Hodge structure and period domain}\label{pd}
A Hodge structure of weight $w$ with Hodge numbers $(h^{p,q})_{p,q}$ is a pair $(H,F)$ consisting of a free $\ZZ$-module $H$ of rank $\sum_{p,q}h^{p,q}$ 
and a decreasing filtration $F$ on $H_{\CC}:=H\otimes \CC$ satisfying the following conditions:
\begin{enumerate}
\item $\dim_{\CC} F^p=\sum_{r\geq p}h^{r,w-r}\quad \text{for all $p$;}$\label{check1}
\item $H_{\CC}=\bigoplus _{p+q=w} H^{p,q}\quad(H^{p,q}:=F^p\cap \overline{F^{w-p}}, \ \dim_\CC H^{p,q}=h^{p,q}).$
\end{enumerate}
For Hodge structures $(H,F)$ and $(H',F')$, a homomorphism $f:H\to H'$ is a $(r,r)$-morphism of the Hodge structures if $f(F^p)\subset F'^{p+r}$ and $f(\bar{F}^p)\subset \bar{F'}^{p+r}$.

A polarization $\langle* ,**\rangle$ for a Hodge structure $(H,F)$ of weight $w$ is a non-degenerate bilinear form on $H$, 
symmetric if $w$ is even and skew-symmetric if $w$ is odd, satisfying the following:
\begin{enumerate}
\item[(3)] $\langle F^p , F^q \rangle=0\quad \text{for $p+q>w$;}$\label{check2}
\item[(4)] $\sqrt{-1}^{p-q}\langle v,\bar{v}\rangle >0$ for $0\neq v\in H^{p,q}$.
\end{enumerate}
 
We fix a polarized Hodge structure $(H_{0},F_0,\langle* ,**\rangle_0)$ of weight $w$ with Hodge numbers $(h^{p,q})_{p,q}$.
We define the period domain $D$ which parametrizes all the Hodge structures of this type by 
$$
D:=\left\{\begin{array}{l|l}F&\begin{array}{r}(H_{0} , F ,\langle* ,**\rangle_0)\text{ is a polarized Hodge structure}\\ \text{ of weight $w$ with Hodge numbers $(h^{p,q})_{p,q}$}\end{array}\end{array}\right\}.
$$
The compact dual $\check D$ of $D$ is 
$$
\check{D}:=\left\{\begin{array}{l|l}F&\begin{array}{r}(H_{0} , F ,\langle* ,**\rangle_0)\text{ satisfies the above (1)--(3)}\end{array}\end{array}\right\}.
$$
Let $G_A:=\Aut{(H_{0}\otimes A,\langle* ,**\rangle_0)}$ for a $\ZZ$-module $A$.
Then $G_{\RR}$ acts transitively on $D$ and $G_{\CC}$ acts transitively on $\check D$. 
A variation of Hodge structure (VHS) over a complex manifold $S$ is a pair $(\calH,\calF)$ consisting of a $\ZZ$-local system and a filtration of $\calH\otimes \calO_S$ over $S$ satisfying the following: 
\begin{enumerate}
\item The fiber $(H_{s},F_{s})$ at $s\in S$ is a Hodge structure;
\item $\nabla\calF^p\subset \calF^{p-1}\otimes \Omega ^1_S$ for the connection 
$\nabla:=id\otimes d: \calH\otimes \calO_S \to \calH\otimes \Omega^1_S$.  
\end{enumerate}
A polarization for a VHS is a bilinear form on the local system which defines a polarization on each fiber.  
In this article, a VHS is always assumed to be polarized. 

For a VHS over $S$, we fix a base point $s_0\in S$.
Let $D$ be the period domain for the Hodge structure at $s_0$.
We then have the period map $\phi:S\to \Gamma\bs D$ via $s\mapsto F_s$, where $\Gamma$ is the monodromy group.


\subsection{Limit mixed Hodge structure}\label{LMHSsec}
Let $\bar{S}$ be a smooth compactification of $S$ such that $\bar{S}- S$ is a normal crossing divisor.
For each $p\in \bar{S} -S$, there exists a neighbourhood $V$ of around $p$ in $\bar{S}$ 
such that $U:=V\cap S\cong (\Delta^*)^m\times\Delta^{n-m}$ where $\Delta$ is the unit disk. 
We can lift the period map to $\tilde{\phi}:\tilde{U}\to D$, where $\tilde{U}\to U$ is the universal covering map. 
Under the identification $\tilde{U}\cong \mathbb{H}^m\times \Delta^{n-m}$, the covering map $\tilde{U}\to U$ is given by 
$$
(z_1,\ldots ,z_{n})\mapsto (\exp{(2\pi \sqrt{-1}z_{1})},\ldots ,\exp{(2\pi \sqrt{-1}z_{m})},z_{m+1},\ldots ,z_{n}). 
$$   
Let $T_1,\ldots ,T_m$ be a generator of the monodromy around $p$ such that
$$
\tilde{\phi}(\cdots,z_j+1,\cdots)=T_j\tilde{\phi}(\cdots,z_j,\cdots).
$$  
Let us assume $T_j$ is unipotent.
Then $N_j=\log{T_j}$ is nilpotent in the Lie algebra $\frakg_{\QQ}$, and $N_1,\ldots ,N_m$ are commutating with each other. 
We define a map $\tilde{\psi}:\tilde{U}\to \check{D}$ by 
$$
z\mapsto\exp{(-\sum_j z_jN_j)}\phi(z).
$$
Since $\tilde{\psi}(\cdots,z_j+1,\cdots)=\psi(\cdots,z_j,\cdots)$, the map $\tilde{\psi}$ descends to a map $\psi:U\to \check{D}$, 
which admits a unique extension to $\psi:\Delta^n\to \check{D}$ by \cite{Sch}. 
We call $F_{\infty}:=\psi(\mathbf{0})\in\check{D}$ the limit Hodge filtration (LHF).

\begin{rem}\label{LMHF}
The LHF is not uniquely determined by a VHS.  
In fact, for $f_j\in \calO_{\Delta}$, we obtain a new coordinate 
$$
(\exp{(2\pi \sqrt{-1}f_1(z_1))}z_1,\ldots ,\exp{(2\pi \sqrt{-1}f_n(z_n))}z_n), 
$$
with respect to which, the LHF is given by $\exp{(-\sum f_j(0)N_j)}F_{\infty}$. 
Moreover, $N_1,\ldots ,N_m$ depend also on the choice of coordinates.
However the nilpotent orbit (to be discussed in the next subsection) is uniquely determined by the VHS.  
\end{rem}

Let $N:=N_1+\cdots +N_m$.  
By \cite{Sch},  we have an increasing filtration $W(N)$ of $H_{\RR,0}:=H_{0}\otimes \RR$. 
Denoting by $W$ the shifted filtration of $W(N)$ by the weight $w$, we observe the pair $(W,F_{\infty})$ has the following properties:
\begin{enumerate}
\item the graded quotient $(\gr^{W}_k,F_{\infty}\gr^{W}_{k,\CC})$ is a Hodge structure of weight $k$;
\item $N$ defines a $(-1,-1)$-morphism $(\gr^{W}_{k},F_{\infty}\gr^{W}_{k,\CC})\to (\gr^{W}_{k-2},F_{\infty}\gr^{W}_{k-2,\CC})$ of Hodge structures;
\item $N^k:(\gr^{W}_{w+k},F_{\infty}\gr^{W}_{w+k,\CC})\to (\gr^{W}_{w-k},F_{\infty}\gr^{W}_{w-k,\CC})$ is isomorphism;
\item the paring $\langle * ,N^k(**)\rangle$ gives a polarization on $(\gr^{W}_{w+k},F_{\infty}\gr^{W}_{w+k,\CC})$.
\end{enumerate}
The pair $(W,F_{\infty})$ is called the limit mixed Hodge structure (LMHS).


\subsection{Partial compactification of period domain}\label{cpt_pd}
We call $\sigma\subset \frakg_{\RR}$ a nilpotent cone if it satisfies the following:
\begin{enumerate}
\item $\sigma$ is a closed cone generated by finitely many elements of $\frakg_{\QQ}$;
\item $N\in \sigma$ is a nilpotent as an endmorphism of $H_{\RR}$;
\item $NN'=N'N$ for any $N,N'\in \sigma$.
\end{enumerate} 
A nonempty set $\Sigma$ of (finitely or infinitely many) cones is called a fan if it satisfies the following: 
\begin{enumerate}
\item If $\sigma\in\Sigma$, any face of $\sigma$ belongs to $\Sigma$;
\item If $\sigma,\tau\in \Sigma$, then $\sigma\cap\tau$ is a face of $\sigma$ and of $\tau$.
\end{enumerate} 
For $A=\RR,\CC$, we denote by $\sigma_{A}$ the $A$-linear span of $\sigma$ in $\mf{g}_A$.
\begin{dfn}\label{nilp}
Let $\sigma=\sum_{j=1}^n\RR_{\geq 0}N_j$ be a nilpotent cone and $F\in\check{D}$. 
Then the pair consisting of $\sigma $ and $\exp{(\sigma_{\CC})}F\subset\check{D}$ is called a nilpotent orbit if it satisfies the following: 
\begin{enumerate}
\item $\exp{(\sum_j \sqrt{-1} y_jN_j)}F\in D$ for all $y_j\gg 0$.
\item $NF^p\subset F^{p-1}$ for all $p\in\ZZ$ and for all $N\in \sigma$.
\end{enumerate}
\end{dfn}
The data $(N_1,\ldots ,N_m,F_{\infty})$ given in the previous section generates a nilpotent orbit.
Moreover, any nilpotent orbit generates a LMHS.
In fact, $W(N)=W(N')$ for any $N$ and $N'$ in the relative interior of $\sigma$ (see \cite{CK} for example),  
and the pair $(W(N)[w],F')$ is a LMHS for any $F'\in \exp{(\sigma_{\CC})}F$.

Let $\Sigma$ be a fan consisting of nilpotent cones. 
We define the set of nilpotent orbits
$$
D_{\Sigma}:=\{(\sigma,Z)|\; \sigma\in\Sigma,\; (\sigma, Z)\;\text{is a nilpotent orbit}\}.
$$
For a nilpotent cone $\sigma$, the set of faces of $\sigma$ is a fan, and we abbreviate $D_{\{\text{faces of }\sigma\}}$ as $D_{\sigma}$.
Let $\Gamma$ be a subgroup of $G_{\ZZ}$ and $\Sigma$ a fan of nilpotent cones.
We say $\Gamma$ is compatible with $\Sigma$ if $\mathrm{Ad}(\gamma)(\sigma)\in\Sigma$ for all $\gamma\in\Gamma$ and $\sigma\in\Sigma$.
Then $\Gamma$ acts on $D_{\Sigma}$ if $\Gamma$ is compatible with $\Sigma$.
Moreover we say $\Gamma$ is strongly compatible with $\Sigma$ if it is compatible with $\Sigma$ 
and for all $\sigma\in \Sigma$ there exist $\gamma_1,\ldots,\gamma_n\in \Gamma(\sigma):=\Gamma\cap\exp{(\sigma)}$ such that $\sigma=\sum_{j}\RR_{\geq 0}\log{(\gamma_j)}$. 

We consider the geometric structure of $\GsDs$ in the case where $\sigma$ has rank $1$ (we will discuss this case in the next section). 
For a nilpotent cone $\sigma=\RR_{\geq 0}N$ and the $\ZZ$-subgroup $\Gamma(\sigma)^{\gp}=e^{\ZZ N}$, we have the partial compactification $\Gamma(\sigma)^{\gp}\bs D_{\sigma}$.
We now show its geometric structure following the exposition of \cite{KU}. 
Let us define
\begin{equation*}
\CC\times \check{D}\supset E_{\sigma}:=\left\{\begin{array}{l|l}(s,F) &\begin{array}{l}\exp{\l(\ell(s) N\r)}F\in D\text{ if } s\neq 0,\\ (\sigma, \exp{(\CC N)}F) \text{ is a nilpotent orbit if }s=0\end{array}\end{array}\right\}, 
\end{equation*}
where $\ell(s)$ is a branch of $\frac{\log(s)}{2\pi \sqrt{-1}}$.
Here $\CC$ is endowed with a log structure as a toric variety and $\CC\times \check{D}$ is a logarithmic analytic space.
By \cite[Theorem A]{KU}, the subspace $E_{\sigma}$ is a log manifold with the map 
$$
E_{\sigma}\to \Gamma(\sigma)^{\mathrm{gp}}\backslash D_{\sigma}, \ \ \
(s,F)\mapsto \begin{cases}\exp{(\ell(s)N)}F & \text{if }s\neq 0,\\(\sigma,\exp{(\sigma_{\CC})}F) & \text{if }s=0.\end{cases}
$$
The geometric structure of $\GsDs$ is induced by the above map, which is a $\C$-torsor, i.e. $\GsDs \cong E_{\sigma}/\C$. 

\begin{thm}[{\cite[Theorem A]{KU}}]
Let $\Sigma$ be a fan of nilpotent cones and let $\Gamma$ be a subgroup of $G_{\ZZ}$ which is strongly compatible with $\Sigma$. 
Then the following hold:
\begin{enumerate}
\item If $\Gamma$ is neat (i.e., the subgroup of $\gm (\CC)$ generated by all the eigenvalues of all $\gamma\in\Gamma$ is torsion free), 
then $\Gamma\backslash D_{\Sigma}$ is a logarithmic manifold.
\item The map $\gsds\to\Gamma\bs D_{\Sigma}$ is open and locally an isomorphism of logarithmic manifolds.
\end{enumerate}
\end{thm}
Logarithmic manifolds are a generalization of analytic spaces introduced in \cite{KU}.
A logarithmic manifold is a subspace of a logarithmic analytic space, whose topology is induced by the strong topology. 

For a VHS, locally the period map $U\to\Gamma\bs D$ can be extended to the map $V\to\Gamma\bs D_{\sigma}$.
We assume that there exists a fan $\Sigma$ which includes all the nilpotent cones coming from the local monodromies around the boundary components $\bar{S} - S$. 
It is worth noting that a construction of such a fan is still an open problem in general, but later we will deal with the one-dimensional case where the construction is straightforward (cf. \cite[\S 4]{U1}). 
Then we have an extended period map $\bar{S}\to \Gamma\bs D_{\Sigma}$ and, although the target space is not an analytic space, we have the following:  
\begin{thm}[{\cite[\S 5]{U1}}]\label{image}
The image of $\bar{S}$ is a compact analytic space if $\bar{S}$ is compact. 
\end{thm}
Moreover, the map is also analytic since the category of logarithmic analytic spaces is a full subcategory of $\mathcal{B}(\mathrm{log})$ 
whose objects are logarithmic manifolds (\cite[\S 3]{KU}).


\section{The case where $\rank{H}=4$ with $h^{3,0}=h^{2,1}=1$} \label{Ch Monodromy}
In this section, we consider the Hodge structures with Hodge numbers 
$$
h^{p,q}=1\ \text{ if } p+q=3, \  p,q\geq 0,\ \ \ \text{and} \ \ \ h^{p,q}=0 \ \text{ otherwise}.
$$
In this case, the partial compactifications of the period domain $D$ are well-studied in \cite[\S 12.3]{KU} and \cite{GGK}. 
Note that $\rank{H}=4$ and $G_{\ZZ}=Sp(2,\ZZ)$. 
The period domain $D$ is the flag domain $Sp(2,\RR)/(U(1)\times U(1))$ of dimension $4$. 
If $\sigma$ generates a nilpotent orbit, then $\sigma=\RR_{\geq 0} N$ and $N$ is one of the following types:
\begin{enumerate}
\item $N^2=0$ and $\dim\Im{N}=1$;
\item $N^2=0$ and $\dim\Im{N}=2$;
\item $N^3\neq 0$ and $N^4=0$.
\end{enumerate}
The case (3) is called maximally unipotent monodromy (MUM).
The goal of this section is to analyze MUM and their LMHS in detail.


\subsection{Normalization of monodromy matrix}
Let $T\in G_{\ZZ}$ be a unipotent element such that $\log{T}=N$ is a MUM element. 
The monodromy weight filtration $W=W(N)[3]$ is
$$
\{0\}=W_{-1}\subset W_0=W_1\subset W_2=W_3\subset W_4=W_5\subset W_6=H_{\QQ} 
$$
with the graded quotient $\Gr^{W}_{2p}\cong \QQ$ for $0 \le p \le 3$. 
The pair $(\Gr_{2p}^W,F\Gr_{2p,\CC}^W)$ is the Tate Hodge structure of weight $2p$ if $(N,F)$ generates a nilpotent orbit.
The LMHS condition induces
$$
\Gr^W_{6}\stackrel{N}{\longrightarrow} \Gr^W_{4}\stackrel{N}{\longrightarrow}
\Gr^{W}_2\stackrel{N}{\longrightarrow}\Gr^W_{0},
$$
where each $N:\Gr^{W}_{2p}\to\Gr^{W}_{2p-2}$ is an isomorphism of Hodge structures.

By \cite[Lemma (I.B.1) \& (I.B.3)]{GGK}, we may choose a symplectic basis $e_0,\ldots ,e_3$ of $H_{\ZZ}$ which satisfies
\begin{align}\label{pol} 
W_{2p}=\Span_{\RR}\{e_{j}\;|\; 0\leq j\leq p \} \ \ (0 \le p \le 3), \ \ \ 
(\langle e_i,e_j\rangle )_{i,j}
=\begin{bmatrix}
0&0&0&1\cr 
0&0&1&0\cr 
0&-1&0&0\cr -1
&0&0&0
\end{bmatrix}.
\end{align}
By \cite[(I.B.7)]{GGK}, with respect to this basis, $N$ is of the form 
\begin{align}\label{norm_N}
N=\begin{bmatrix}0&0&0&0\cr a&0&0&0\cr e&b&0&0\cr f&e&-a&0\end{bmatrix}.
\end{align}
for some $a,b,e,f \in \Q$. 
The polarization condition of a LMHS yields the following inequalities: 
$$
\sqrt{-1}^{6}\langle e_3, N^3 e_3\rangle =a^2b>0,\quad \sqrt{-1}^4\langle e_2, N e_2\rangle = b>0. 
$$
Moreover, we have 
$$
T=e^N=\begin{bmatrix}1&0&0&0\cr a&1&0&0\cr e+\frac{ab}{2}&b&1&0\cr f-\frac{a^2b}{6}&e-\frac{ab}{2}&-a&1\end{bmatrix}\in G_{\ZZ},
$$
which shows that 
\begin{align}\label{abef}
a,\ b,\ e \pm \frac{ab}{2}, \ f-\frac{a^2b}{6}\in \ZZ.
\end{align}
The symplectic basis $e_3,\ldots ,e_0$ with the property (\ref{pol}) is not unique; 
for any $A\in G_{\ZZ}(W):=\Aut(H,\langle *,**\rangle,W)$, the new basis $Ae_3,\ldots ,Ae_0$ will do. 
Any $A\in G_{\ZZ}(W)$ is represented by a lower triangular matrix with $1$'s on the diagonal, and thus written as $A=e^M$ for
$$
M=\begin{bmatrix}0&0&0&0\cr p&0&0&0\cr r&q&0&0\cr s&r&-p&0\end{bmatrix}
$$
where $p,q,r,s$ satisfy the same condition as $a,b,e,f$ in (\ref{abef}). 
Under the transformation $N\to \Ad{(A)}N$, the entries $a,b,e,f$ change as follows: 
\begin{align}\label{trans}
&a\mapsto a,\quad b\mapsto b,\quad e\mapsto e-bp+aq,\\
&f\mapsto f-2ep+bp^2-apq+2ar.\nonumber
\end{align}

\begin{prop}\label{GGK_IB10}
Under the action of $G_{\ZZ}(W)$, $b$ is invariant, and $a$ is invariant up to $\pm 1$.
Moreover, for $m=\mathrm{gcd}(a,b)$, $[e]\in \ZZ / m\ZZ$ is invariant if $ab$ is even, and $[2e]\in \ZZ / 2m\ZZ$ is invariant if $ab$ is odd.
\end{prop} 
\begin{proof}
See the Proposition I.B.10 of \cite{GGK}. 
\end{proof}


\subsection{Period map around boundary point}\label{pm_bd}
Let $(\calH,\calF)$ be a VHS over the punctured unit disk $\Delta^*$ with monodromy $N$ of the form (\ref{norm_N}). 
Hereafter, we fix such a presentation with $a,b,e,f$. 
For the monodromy group $\Gamma=\langle T \rangle$, we have the period map $\phi:\Delta^*\to\Gamma\bs D$ and its lifting $\tilde{\phi}:\mathbb{H}\to D$. 
Now $\exp{(-zN)}\tilde{\phi}(z)$ descends to a holomorphic map on $\Delta$, we denote it by $\psi(s)$ where $s=\exp{(2\pi \sqrt{-1}z)}$.
Here $F_{\infty}=\psi(0)$ is the LHF and then $F_{\infty}^3\cap \overline{F_{\infty}^3}\pmod{W_5}$ is generated by $e_3$.
We may choose a generator
$$
e_3+\pi_2e_2+\pi_1e_1+\pi_0e_0
$$
of the subspace $F^3_{\infty}$ for some $\pi_2,\pi_1,\pi_0\in\CC$.
Then the subspace $F_{\psi(s)}^3$ corresponding to $\psi(s)\in\check{D}$ is generated by 
$$
\psi_3(s)e_3+\psi_2(s)e_2+ \psi_1(s)e_1+\psi_0(s)e_0
$$
where $\psi_i$ for $0 \le i \le 3$ are some holomoprhic functions on $\Delta$ with $\psi_3(0)= 1$ and $\psi_i(0)= \pi_i$ for $0 \le i \le 2$.  
By untwisting $\psi$, a local frame of the subspace $\calF^3$ spanned by the period is given by 
$$\begin{bmatrix}\omega_3(s)\cr\omega_2(s)\cr\omega_1(s)\cr\omega_0(s)\end{bmatrix}:=\exp{(zN)}\begin{bmatrix}\psi_3(s)\cr\psi_2(s)\cr\psi_1(s)\cr\psi_0(s)\end{bmatrix}.$$
Here $\omega_3(s)=\psi_3(s)$ and
$$\omega_2(s)=a\omega_3(s)\frac{\log{(s)}}{2\pi \sqrt{-1}}+\psi_2(s).$$
Therefore
\begin{align}\label{q}
q(s):=\exp{\left(2\pi \sqrt{-1}\frac{\omega_2(s)}{a\omega_3(s)}\right)}=\exp{\left(2\pi \sqrt{-1} \frac{\psi_2(s)}{a\psi_3(s)}\right)}s
\end{align}
defines a new coordinate of $\Delta$, which is known as the {\it mirror map} in the context of  mirror symmetry. 

By the discussion in \S \ref{cpt_pd}, we have the extended period map $\phi:\Delta\to\Gamma\bs \Ds$, 
and the geometric structure of the image $\phi (\Delta)\subset \Gamma\bs \Ds$ is induced by that of the $\C$-torsor $E_{\sigma}\to\GsDs$.
\begin{lem}\label{surj}
The period map $\phi:\Delta \to \phi (\Delta)$ is an isomorphism as analytic spaces. 
\end{lem}

\begin{proof}
The coordinate $q$ gives a local section of the $\C$-torsor $E_{\sigma}\to\GsDs$ restricted on the image $\phi(\Delta)$.
In fact, we can define 
$$
\rho :\phi(\Delta) \to E_{\sigma};\quad \phi(s)\mapsto (q(s),\exp{\left(-\frac{\psi_2(s)}{a\psi_3(s)}N\right)}\psi(s)). 
$$
This induces isomorphsims $\Delta\cong\rho (\phi(\Delta))\cong\phi(\Delta)$ as analytic spaces. 
\end{proof}
Moreover the map $\Delta \to \phi (\Delta)$ induces an isomorphism of the log structures in a manner similar to \cite[\S 4--5]{U2}.


\subsection{Normalization of LHF}
Let $(\sigma,\exp{(\sigma_{\CC})}F)$ be a nilpotent orbit, i.e. $(W,F)$ is a LMHS.
We show that we have a canonical choice of $F$ which has a normalized form with respect to the symplectic basis $e_3,\ldots ,e_0$.

For the LMHS $(W,F)$, we have the Deligne decomposition $H_{\CC}=\bigoplus_{0\leq j\leq 3}I^{j,j}$ so that
$$
W_{2p}= \bigoplus_{k\leq p} I^{k,k},\quad F^{p}=\bigoplus_{k\geq p} I^{k,k}
$$
for $0 \le p \le 3$.
We can take a unique generator $v_p\in I^{p,p}$ such that $[v_p]=[e_p]$ in $\Gr_{2p,\CC}^W$. 
By \cite[Proposition (I.C.2)]{GGK}, with repect to the basis $v_3,\dots,v_0$, the matrix $N$ is of the form 
$$\begin{bmatrix}0&0&0&0\cr a&0&0&0\cr 0&b&0&0\cr 0&0&-a&0\end{bmatrix}.$$
Moreover, by \cite[Proposition (I.C.4)]{GGK}, the period matrix of $F$ is then written as
\begin{align}\label{pm}
\begin{bmatrix}1&0&0&0\cr \pi_2&1&0&0\cr \pi_1&\frac{b}{a}\pi_2+\frac{e}{a}&1&0\cr \pi_0&\frac{e}{a}\pi_2+\frac{f}{a}-\pi_1&-\pi_2&1\end{bmatrix}. 
\end{align}
By multiplying $F$ by $\exp{\left(-\frac{\pi_2}{a}N\right)}$, we may assume that $\pi_2=0$. 
For $\pi_2=0$, by the second bilinear relation \cite[(I.C.10)]{GGK}, the period matrix (\ref{pm}) is of the form
\begin{align}\label{norm_pm}
\begin{bmatrix}
1&0&0&0\cr 
0&1&0&0\cr 
\frac{f}{2a}&\frac{e}{a}&1&0\cr 
\pi_0& \frac{f}{2a}&0&1\end{bmatrix}.
\end{align}
Here the values $\frac{f}{2a}$, $\frac{e}{a}$ and $\pi_0$ correspond to the extension class of the LMHS \cite[\S I.C]{GGK}.
We observe that the boundary component $D_{\sigma}\setminus D \cong \C$ is parametrized by $\pi_0$. 

Recall that the LHF depends on the choice of a coordinate for a VHS (Remark \ref{LMHF}). 
If we use the canonical coordinate $q$ of (\ref{q}), the normalized period matrix takes the form of (\ref{norm_pm}). 
In this case, the LHF is given by
$$
F_{\infty}^3=\lim_{z\to 0}\exp{\left(-\frac{\log{z}}{2\pi \sqrt{-1}}N\right)}F^3_z=\begin{bmatrix}1\cr 0\cr \frac{f}{2a}\cr \pi_0 \end{bmatrix}.
$$


\section{Generic Torelli theorem} \label{Ch Generic Torelli}
The goal of this section is to show the generic Torelli theorem for a class of one-parameter families of Calabi--Yau threefolds.


\subsection{Degree of period map}
Let $\mathscr{X}\rightarrow S$ be a one-parameter family of Calabi--Yau threefolds. 
Given a smooth compactification $\bar{S}$ of $S$ so that $\bar{S}-S$ consists of finite points. 
Let $\phi:S\to \Gamma\bs D$ be the period map associated to the VHS on $H:=H^3(X,\Z)/\mathrm{Tor}$ for a fixed fiber $X$, 
where $\mathrm{Tor}$ denotes the torsion part of $H^3(X,\Z)$. 
Although the monodromy group $\Gamma$ is not necessary a neat subgroup of $G_{\ZZ}$, there always exists a neat subgroup $\Gamma'$ of $\Gamma$ of finite index. 
In this situation, we have a lifting $\tilde{\phi}$ of $\phi$
\begin{align*}
\xymatrix{
\tilde{S}\ar@{->}[d]\ar@{->}[r]^{\tilde{\phi}} & \Gamma'\bs D\ar@{->}[d]\\
S\ar@{->}[r]^\phi & \Gamma\bs D
}
\end{align*}
where $\tilde{S}$ is a finite covering of $S$. 
To show the generic Torelli theorem for $\phi$, it suffices to show the theorem for the lifting $\tilde{\phi}:\tilde{S}\to\Gamma'\bs D$. 
Therefore we henceforth assume that $\Gamma$ is neat. 
We also assume that the Kodaira--Spencer map is an isomorphism on the base curve $S$ to exclude trivial cases \cite{BG}.  
To summarize, we assume that:
\begin{enumerate}
\item the monodromy group $\Gamma$ is neat; 
\item the Kodaira--Spencer map is an isomorphism on $S$. 
\end{enumerate}

We define a fan 
$$
\Xi:=\{\RR_{\geq 0}N\; | \; N\text{ is a nilpotent element in }\mathfrak{g}_{\mathbb{Q}}\}.
$$
In this one-parameter situation, the nilpotent cones are one-dimensional and are contained in $\Xi$. 
In fact, under the condition $h^{3,0}=h^{2,1}=1$, the nilpotent cones which generate nilpotent orbits are contained in $\Xi$, and such a fan is called a complete fan in \cite{KU}.
Except for some special cases such as the one-parameter case or the Hermitian symmetric case, a construction of a complete fan is not known.  

By \cite{KU}, the partial compactification $\Gamma\bs D_{\Xi}$ of $\Gamma\bs D$ is a logarithmic manifold and the period map extends to $\phi:\bar{S}\to \Gamma\bs D_{\Xi }$. 
By Theorem \ref{image}, the image $\phi(\bar{S})$ and the map $\phi$ is analytic.
Moreover $\phi$ is proper and thus a finite covering map.  
\begin{prop}\label{deg}
Let $p\in\bar{S}-S$ be a MUM point. 
If $\phi^{-1}(\phi(p))=\{p\}$, then the map $\phi:\bar{S}\rightarrow \phi(\bar{S})$ is of degree $1$.
\end{prop}
\begin{proof}
By Lemma \ref{surj}, a disk $\Delta_p$ around a MUM point $p$ is isomorphic to the image $\phi(\Delta_p)$.
Since $\phi(p)$ is not a branch point, the map $\phi$ must be of degree $1$.  
\end{proof}

For $p\in\bar{S}-S$, the image $\phi(p) $ is the nilpotent orbit determined by the local monodromy and the LHF around $p$.
If the family has only one MUM, we clearly have $\phi^{-1}(\phi(p))=p$, therefore the generic Torelli theorem holds by Proposition \ref{deg}. 

To show the generic Torelli theorem for a family with multi MUMs, 
it suffices to show that there exists a MUM point $p_1$ such that for any other MUM point $p_2$ the condition $\phi(p_1)\ne \phi(p_2)$ holds. 
Let $N_j$ be the logarithm of the local monodromy around $p_j$, and let $F_j$ be the LHF.
Then 
$$\phi(p_j)=(\sigma_j,\exp{(\sigma_{j,\CC})}F_j)\mod{\Gamma}$$
where $\sigma_{j}=\RR_{\geq 0}N_j$. 
As discussed in the previous section, we have the normalized matrix (\ref{norm_N}) of $N_j$ determined by $a_j,b_j,e_j,f_j\in\QQ$ 
and the canonical choice (\ref{norm_pm}) of $F_j$ determined by $\pi_0^j\in\CC$ using a symplectic basis $e^j_3,\ldots ,e^j_0$ satisfying (\ref{pol}).

\begin{prop} \label{image of two MUM}
If $b_1\neq b_2$ or $\pi_0^1- \pi_0^2 \not \in \Q$, then 
$$g(\sigma_1,\exp{(\sigma_{1,\CC})}F_1)\neq (\sigma_2,\exp{(\sigma_{2,\CC})}F_2)$$
for any $g\in G_{\ZZ}$.
In other words, we have $\phi(p_1)\neq \phi(p_2)$.
\end{prop}

\begin{proof}
We define $g\in G_{\ZZ}$ by $e^1_k\mapsto e^2_k$. 
Then $\Ad{(g)}N_1$ is written as the normalized matrices determined by $a_1,b_1,e_1,f_1$ using the symplectic basis $e^2_3,\ldots ,e^2_0$, and $\Ad{(g)}W(N_1)=W(N_2)$.
We put $W=W(N_2)$.
If $b_1\neq b_2$, there does not exist $h\in G_{\ZZ}(W)$ such that $\Ad{(hg)}N_1\in\sigma_2$ 
since $b_1$ is invariant under the action of $G_{\ZZ}(W)$ by Proposition \ref{GGK_IB10}. 
Then $\Ad(\gamma)\sigma_1\neq \sigma_2\mod{\Gamma}$ for any $\gamma\in G_{\ZZ}$.

Now suppose that $b_1= b_2$ and that there exists $h\in G_{\ZZ}(W)$ such that $\Ad{(hg)}\sigma_1=\sigma_2$. 
The filtration $gF_1$ is written as the normalized period matrix determined by $\pi_0^1$ using $e^2_3,\ldots ,e^2_0$.
Then the period matrix of the canonical choice in 
$$hg\exp{(\sigma_{1,\CC})}F_1=\exp{(\sigma_{2,\CC})}hgF_1$$
is determined by $\pi_0^1+\lambda$ with $\lambda\in\QQ$ since $h\in G_{\ZZ}$ and $N_2\in \frakg_{\QQ}$.
Since $\pi_0^1- \pi_0^2 \not \in \Q$, we conclude that $hg\exp{(\sigma_{1,\CC})}F_1\neq\exp{(\sigma_{2,\CC})}F_2$. 
Therefore there does not exist $\gamma\in G_{\ZZ}$ such that $\Ad(\gamma)\sigma_1=\sigma_2$ and $\gamma\exp{(\sigma_{1,\CC})}F_1\neq\exp{(\sigma_{2,\CC})}F_2$.
\end{proof}

\begin{thm}[Generic Torelli Theorem] \label{generic}
Let $\mathscr{X}\rightarrow S$ be a one-parameter family of Calabi--Yau threefolds with a MUM point.  
Assume that there exists a MUM point $p_1$ such that for any other MUM point $p_2 \in \bar{S}-S$ the condition $b_1\neq b_2$ or $\pi_0^1- \pi_0^2 \not \in \Q$ holds. 
Then the map $\phi:\bar{S}\rightarrow \phi(\bar{S})$ is the normalization of $\phi(\bar{S})$. 
\end{thm}
\begin{proof}
The assertion readily follows from the combination of Proposition \ref{image of two MUM} and Proposition \ref{deg} as discussed above. 
\end{proof}

Theorem \ref{generic} in particular applies to the families of Calabi--Yau threefolds with exactly one MUM point. 
Such examples include almost all one-parameter mirror families of complete intersection Calabi--Yau threefolds in weighted projective spaces and homogeneous spaces \cite{vEvS}.   
We will discuss some Calabi--Yau threefolds with two MUM points in the next section.


\section{Mirror Symmetry}\label{Ch MS}
In this section, we see that the Hodge theoretic invariants $b$ and $\pi_0$ appear in the framework of mirror symmetry. 
Mirror symmetry claims, given a family of Calabi--Yau threefolds $\mathscr{X}\to B$ with a MUM point, 
there exists another family $\mathscr{X}^{\vee} \to B^\vee$ of Calabi--Yau threefolds such that 
some Hodge theoretic invariants of $X$ around the MUM point and symplectic invariants of $X^\vee$ are equivalent in a certain way.  
Here $X$ and $X^\vee$ are generic members of $\mathscr{X} \to B$  and $\mathscr{X}^{\vee} \to B^\vee$ respectively. 
Simply put, mirror symmetry interchanges the complex geometry of a Calabi--Yau threefold $X$ with the symplectic geometry of another, called a mirror threefold $X^\vee$, 
and such a correspondence depends on the choice of a MUM point. 
We should think that each MUM point corresponds to a mirror threefold.  
For instance, if a family $\mathscr{X}\to B$ has several MUM points, there should be several mirror threefolds. 
We refer the reader to \cite{CK2} for a detailed treatment of mirror symmetry.

We now investigate the interplay between the LMHS at a MUM point and the corresponding mirror threefold, 
restricting ourselves to one-parameter models i.e. the case when $h^{2,1}(X)=h^{1,1}(X^{\vee})=1$. 
Since the complex moduli space of $X$ is 1-dimensional, $X$ comes with a family $\mathscr{X}\to S$ over a punctured curve $S$. 
Since mirror symmetry is a statement about a MUM point, we assume that such a puncture point of $S$ corresponding to $X^\vee$ is chosen.

We denote by $\Omega_z$ a holomorphic 3-form on the Calabi--Yau threefold $X_z$ over a point $z \in S$ of the family $\mathscr{X}\to S$. 
On an open disk $\Delta$ around the MUM point $z=0$, there exist solutions $\omega_0,\dots,\omega_3$ of the Picard--Fuchs equation of the following form \cite{CK2}: 
\begin{align} \label{normalized sols}
\omega_3(z) &=  \psi_3(z) =1+O(z),\\ 
2\pi \sqrt{-1}\omega_2(z) &= \psi_3(z)\log(z) + \psi_2(z),\notag \\
(2\pi \sqrt{-1})^2\omega_1(z) &= 2\psi_2(z)\log(z) + \psi_3(z)\log(z)^2 + \psi_1(z), \notag\\
(2\pi \sqrt{-1})^3\omega_0(z) &= 3\psi_1(z)\log(z) + 3\psi_2(z)\log(z)^2 + \psi_3(z)\log(z)^3 +\psi_0(z), \notag
\end{align}
where $\psi_j$ is a power series in $z$ such that $\psi_j(0)=0$ for $0\leq j\leq 2$. 
An important observation is that the local monodromy group at each MUM point is often controlled by the topological invariants of the mirror threefold $X^\vee$ as follows.  
Let $z_0 \in \Delta^*$ be a reference point. 
We equip $H^3(X_{z_0},\Z)/\mathrm{Tor}$ with the standard symplectic form, 
and then mirror symmetry predicts the existence of a symplectic basis $A_0,A_1,B^1,B^0$ of $H_3(X_{z_0},\Z)/\mathrm{Tor}$ such that 
 \begin{align}\label{sympl_basis}
 &  \left[
\begin{array}{cccc}
\int_{A_0}\Omega_z\\
\int_{A_1}\Omega_z\\
\int_{B^1}\Omega_z \\
\int_{B^0}\Omega_z \\
\end{array}
\right]
 = \left[
\begin{array}{cccc}
1 & 0 & 0 & 0\\
0 & 1 & 0 &0 \\
-\frac{c_2(X^\vee)\cdot H}{24} & \lambda & \frac{\deg{X^\vee}}{2} & 0\\
\frac{\zeta(3)\chi(X^\vee)}{(2\pi \sqrt{-1})^3} &  -\frac{c_2(X^\vee)\cdot H}{24} &  0 & -\frac{\deg{X^\vee}}{6} \\
\end{array}
\right]
  \left[
\begin{array}{cccc}
\omega_3(z)\\
\omega_2(z)\\
\omega_1(z)\\
\omega_0(z)\\
\end{array}
\right],
\end{align}
where $\lambda=1$ if $\deg{X^\vee}$ is even and $=-\frac{1}{2}$ otherwise.  
This observation was first made in \cite{CdOGP}. 
Although it is conjectural in general, it remains true for a large class of Calabi--Yau threefold, for example, those listed in \cite[Table 1]{vEvS}. 
\begin{Prop} \label{top inv}
Assume the relation (\ref{sympl_basis}). Then the normalized matrix (\ref{norm_N}) of $N=\log{T}$ and the normalized period matrix (\ref{norm_pm}) of the LHF are determined by 
\begin{align}
a=1,\ b=\deg{X^\vee}, \ e=\begin{cases}1&\text{if $b$ is even}\cr -\frac{1}{2}&\text{if $b$ is odd}, \end{cases} \
f=-\frac{c_2(X^\vee)\cdot H}{12},\ \pi_0=\frac{\chi(X^\vee)\zeta (3)}{(2\pi \sqrt{-1})^3}.\notag
\end{align}
\end{Prop}
\begin{proof}
The monodromy matrix of $[\omega_3,\omega_2,\omega_1,\omega_0]^T$ can easily read from the normalization condition (\ref{normalized sols}).  
It is not hard to rewrite it with respect to the symplectic basis to obtain $N$. 
The LHF is obtained in a similar manner. 
\end{proof}
Therefore we see that the LMHS reflects the topological invariants of the mirror threefold. 
With this topological interpretation, for instance, the integrality condition (\ref{abef}) is explained by the Riemann--Roch theorem. 

\begin{exm}
For the mirror family of a quintic threefold $X^\vee$, $a,b,e,f$ and $\pi_0$ are computed in \cite[(III.A)]{GGK}: 
$$
a=-1,\ b=5, \ e=\frac{11}{2},\ f=-\frac{25}{6},\ \pi_0=\frac{-200\zeta(3)}{(2\pi \sqrt{-1})^3}.
$$
Here $\deg X^\vee=5$, $c_2(X^\vee)\cdot H=50$ and $\chi(X^\vee)=-200$.
By the base change (\ref{trans}), we may change $a$ and $e$ into $1$ and $-\frac{1}{2}$ respectively.  
\end{exm}


\subsection{Multiple Mirror Symmetry}\label{const_mir}
We find Theorem \ref{generic} and Proposition \ref{top inv} particularly interesting when a family of Calabi--Yau threefolds has two MUM points.
Such a family is of considerable interest because the existence of two MUMs implies the existence of two mirror partners. 
Such a multiple mirror phenomenon has been studied in \cite{Rod, Kan, HT, Miu}.   
Here we will discuss two examples.  

\begin{exm}[Grassmannian and Pfaffian Calabi--Yau threefold] \label{Pfaff-Grass}
The Grassmannian $\Gr(2,7)$ has a canonical polarization via the Pl\"ucker embedding into $\PP^{20}$. 
The complete intersection of 7 hyperplane sections of this embedding
$$
X^\vee:=\Gr(2,7)\cap (1^7) \subset \PP^{20}
$$
is a Calabi--Yau threefold with $h^{1,1}=1$.

On the other hand, let $N$ be a $7\times 7$ skew-symmetric matrix $N = (n_{ij})$ with $[n_{ij}]_{i<j} \in \PP^{20}$. 
The rank $4$ locus forms a codimension $3$ variety $\mathrm{Pfaff}(7)$ in $\PP^{20}$, known as the Pfaffian. 
The complete intersection of 14 hyperplane sections of the Pfaffian variety 
$$
Y^\vee:=\mathrm{Pfaff}(7)\cap (1^{14}) \subset \PP^{20}
$$
is a Calabi--Yau threefold with $h^{1,1}=1$. 

In each case, a mirror family is constructed by an orbifolding method \cite{Rod}. 
A surprising observation is that the mirror families $\mathscr{X}\rightarrow \PP^1$ and $\mathscr{Y}\rightarrow \PP^1$ are identical and have exactly two MUM points: 
one corresponds to $X^\vee$ and the other to $Y^\vee$.
Mirror symmetry for this example was confirmed in \cite{BCFKS,Tjo}. 
Since $\deg X^\vee=42 \ne \deg Y^\vee=14$, Theorem \ref{generic} applies and the geneic Torelli theorem holds for the mirror family. 
An identical argument applies to, for example, the Calabi--Yau threefolds constructed in \cite{HT,Miu}. 
\end{exm}

\begin{exm}[Complete Intersection $\Gr(2,5)\cap \Gr(2,5)\subset \PP^9$]
Let $i_1,i_2:\mathrm{Gr}(2,5)\hookrightarrow \mathbb{P}^{9}$ be generic Pl\"ucker embeddings. 
It is shown in \cite{Kan} that the complete intersection 
$$
X^\vee:=i_{1}(\mathrm{Gr}(2,5))\cap i_{2}(\mathrm{Gr}(2,5))
$$
is a Calabi--Yau threefold with $h^{1,1}=1$. 
In \cite{Miu2}, a mirror family $\mathscr{X}\rightarrow \PP^1$ of $X^\vee$ was constructed via a toric degeneration of $\Gr(2,5)$ to a Hibi toric variety. 
The mirror family has exactly two MUM points, both of which correspond to $X^\vee$. 
The Hodge theoretic invariants around the MUM points are identical and thus Theorem \ref{generic} cannot be applied in this case. 
It is worth mentioning that it is recently shown that $X^\vee$ gives a counterexample to the birational Torelli problem \cite{BCP}. 
\end{exm}


\par\noindent{\scshape \small
School of Business Administration, Senshu University,\\
2-1-1 Higashimita, Tama, Kawasaki, Kanagawa 214-8580, Japan.}
\par\noindent{\ttfamily hayama@isc.senshu-u.ac.jp}
\break
\par\noindent{\scshape \small
Department of Mathematics, Kyoto University, \\
Kitashirakawa-Oiwake, Sakyo, Kyoto, 606-8502, Japan }
\par\noindent{\ttfamily akanazawa@math.kyoto-u.ac.jp}

\end{document}